\documentclass[12pt]{amsart}
\usepackage{amsmath}
\usepackage{amssymb}
\usepackage[all]{xy}
\usepackage{longtable}
\usepackage[osf,sc]{mathpazo}
\usepackage{euscript}

\newtheorem{theorem}[equation]{Theorem}
\newtheorem{lemma}[equation]{Lemma}
\newtheorem{proposition}[equation]{Proposition}
\newtheorem{corollary}[equation]{Corollary}
\newtheorem{conjecture}[equation]{Conjecture}
\newtheorem{definition}[equation]{Definition}

\theoremstyle{remark}
\newtheorem{remark}[equation]{Remark}

\numberwithin{equation}{subsection}

\newcommand{\DD}{\mathbb{D}}
\newcommand{\FF}{\mathbb{F}}
\newcommand{\ZZ}{\mathbb{Z}}
\newcommand{\QQ}{\mathbb{Q}}

\newcommand{\TT}{\mathbb{T}}
\newcommand{\GG}{\mathbb{G}}

\newcommand{\CC}{\mathbb{C}}

\newcommand{\NN}{\mathbb{N}}

\newcommand{\bff}{\mathbf{f}}
\newcommand{\bg}{\mathbf{g}}

\newcommand{\bu}{\mathbf{u}}
\newcommand{\bv}{\mathbf{v}}

\newcommand{\bF}{\mathbf{F}}
\newcommand{\bG}{\mathbf{G}}

\DeclareMathOperator{\Mat}{Mat}

\DeclareMathOperator{\Li}{Li}

\newcommand{\ok}{\overline{k}}

\newcommand{\power}[2]{{#1 [\![ #2 ]\!]}}
\newcommand{\laurent}[2]{{#1 (\!( #2 )\!)}}

\begin{document}

\title[Linear independence of monomials of multizeta values]{Linear independence of monomials of multizeta values in positive characteristic}

\author{Chieh-Yu Chang}
\address{Department of Mathematics, National Tsing Hua University and National Center for Theoretical Sciences, Hsinchu City 30042, Taiwan
  R.O.C.}

\email{cychang@math.nthu.edu.tw}

\thanks{ The author was partially supported by Golden-Jade fellowship of Kenda Foundation, NCTS and NSC
Grant 100-2115-M-007-010-MY3}
\thanks{}


\date{March 13, 2014}
\begin{abstract}
 In this paper, we study transcendence theory for Thakur multizeta values in positive characteristic. We prove an analogue of the strong form of Goncharov's conjecture. The same result is also established for Carlitz multiple polylogarithms at algebraic points.
\end{abstract}


\maketitle
\section{Introduction}
\subsection{Classical multiple zeta values}
Multiple zeta values (abbreviated as MZVs) are real numbers defined by Euler:
\[ \zeta(s_{1},\ldots,s_{r}):=\sum_{n_{1}>\cdots>n_{r}\geq 1}\frac{1}{n_{1}^{s_{1}}\cdots n_{r}^{s_{r}}  },\]
where $s_{1},\ldots,s_{r}$ are positive integers with $s_{1}\geq 2$. Here $r$ is called the depth and $\sum_{i=1}^{r}s_{i}$ is called the weight of the MZV $\zeta(s_{1},\ldots,s_{r})$. These values are generalizations of the Riemann zeta function at positive integers, and have been much studied in recent years because of various points of view of their interesting properties. For example, they occur as periods of the mixed Tate motives, and they occur as values of Feynman integrals in quantum field theory. We refer the reader to the papers on this subject by Brown, Deligne, Drinfeld, Goncharov, Hoffman, Kaneko, Terasoma,  Zagier etc. See also the recent advances by Brown~\cite{B12} and Zagier~\cite{Z12}.

 It is natural to ask the transcendence nature of these MZVs. However, it is still an open problem although one knows the transcendence of the Riemann zeta function at even positive integers because of Euler's formula. Let $\mathfrak{Z}$ be the $\QQ$-algebra generated by all MZVs and for $w\geq2$ let $\mathfrak{Z}_{w}$ be the $\QQ$-vector space spanned by the weight $w$ MZVs. It is well known that $\mathfrak{Z}_{w_{1}}\mathfrak{Z}_{w_{2}}\subset \mathfrak{Z}_{w_{1}+w_{2}}$ for $w_{1}\geq2$, $w_{2}\geq2$. The main motivation of the study in this paper is from the important  conjecture given by Goncharov \cite{G97}: $\mathfrak{Z}$ forms a graded algebra (graded by weights), i.e.,  $\mathfrak{Z}=\QQ \oplus_{w\geq 2} \mathfrak{Z}_{w}$. The following conjecture (folklore) is a stronger form of Goncharov's conjecture.
\begin{conjecture}\label{Conj:introd}
Let $\overline{\mathfrak{Z}}$ be the $\overline{\QQ}$-algebra generated by MZVs, and let $\overline{\mathfrak{Z}}_{w}$ be the $\overline{\QQ}$-vector space spanned by the weight $w$ MZVs for $w\geq 2$. Then one has that
\begin{enumerate}
\item $\overline{\mathfrak{Z}}$ forms a graded algebra, i.e.,
$\overline{\mathfrak{Z}}=\overline{\QQ}\oplus_{w\geq 2}\overline{\mathfrak{Z}}_{w}$.
\item $\overline{\mathfrak{Z}}$ is defined over $\QQ$ in the sense that the canonical map $\overline{\QQ}\otimes_{\QQ}\mathfrak{Z}\rightarrow \overline{\mathfrak{Z}}$ is bijective.
\end{enumerate}
\end{conjecture}In other words, to prove their linear independence over $\Bbb Q$, one could adopt a strategy of  proving linear independence over $\overline{\Bbb Q}$ for these special values although it is still wild open. The primary purpose of this article is to  prove an analogue of the conjecture above in the setting of multizeta values in positive characteristic.

\subsection{Thakur multizeta values}
In analogy with the classical MZVs, in his seminal work \cite{Thakur04} Thakur studied the characteristic $p$ multizeta values (abbreviated as MZVs) in $k_{\infty}^{\times}$, where $k$ is the rational function field $\FF_{q}(\theta)$ over a finite field $\FF_{q}$ and $k_{\infty}$ is the completion of $k$ at $\infty$, which is the zero divisor of $1/\theta$. Let $A_{+}$ be the monic polynomials in $A$. For any r-tuple $\mathfrak{s}=(s_{1},\ldots,s_{r})\in \NN^{r}$, the multizeta value at $\mathfrak{s}$ is defined by the series
\[\zeta_{A}(\mathfrak{s}):=\sum\frac{1}{a_{1}^{s_{1}}\cdots a_{r}^{s_{r}}}\in k_{\infty}  ,\]
where the sum is over $(a_{1},\ldots,a_{r})\in A_{+}^{r}$ with $\deg_{\theta} a_{i}$ strictly decreasing. Note that $\zeta_{A}(\mathfrak{s})$ is non-vanishing by the work of Thakur~\cite{Thakur09a}.

These MZVs are generalizations of the Carlitz zeta values at positive integers \cite{Carlitz}, and they occur as periods of mixed Carlitz-Tate motives (explicitly constructed) by the  work of Anderson-Thakur \cite{AT09}.  Notice that the weight one MZV is just the Carlitz zeta value at $1$, which exists in this non-archimedean field setting. Note further that Thakur \cite{Thakur10} showed that a product of two multizeta values of weight $w_{1}$ and $w_{2}$ can be expressed as an  $\FF_{p}$-linear combination of MZVs of weight $w_{1}+w_{2}$ (see \cite{LR11} for the explicit expressions).

 Our main result in this paper is to prove a precise function field analogue of Conjecture~\ref{Conj:introd} (stated as Theorem~\ref{T:Main Thm}). That is,  the $\bar{k}$-algebra generated by all MZVs forms a graded algebra (graded by weights) and it is defined over $k$. As consequences, one has:
\begin{enumerate}
\item[$\bullet$] Each nontrivial monomial of MZVs is transcendental over $k$.

\item[$\bullet$] The ratio of two different weight nontrivial monomials of MZVs is transcendental over $k$.
\end{enumerate}
The results above generalize the work of Yu \cite{Yu91, Yu97} for the depth one case, and the work of Thakur \cite{Thakur09b} on the transcendence of some specific MZVs.  We further derive the following consequences stated as Theorem~\ref{T:EulerDichMZV} and Corollary~\ref{Cor:EulerDich}:
\begin{enumerate}
\item[$\bullet$] Let $Z_{1}$ and $Z_{2}$ be  two MZVs of the same weight. Then either $Z_{1}/Z_{2}\in k$ or $Z_{1}$ and $Z_{2}$ are algebraically independent over $k$.
\item[$\bullet$] Let $Z$ be a MZV of weight $w$. Then either $Z/\tilde{\pi}^{w}$ is in $k$ or $Z$ is algebraically independent from $\tilde{\pi}$.
\end{enumerate}Here $\tilde{\pi}$ is a fundamental period of the Carlitz module, which plays the analogous role of $2\pi \sqrt{-1}$ for the multiplicative group $\GG_{m}$. The last property listed above is called {\it{Euler dichotomy}} phenomenon (see \S~\ref{sec:EulerDichMZV}). In particular, every multizeta value of \lq\lq odd \rq\rq weight $w$ (i.e., $(q-1)\nmid w$) is algebraically independent from $\tilde{\pi}$.

The main goal of transcendence theory for MZVs is to determine all the $\bar{k}$-algebraic relations among the MZVs. However, in contrast to the classical case, a nice description of the full set of identities satisfied by MZVs is not known yet (see \cite{AT09,Thakur10}). As all $\bar{k}$-algebraic relations among the MZVs are $\bar{k}$-linear relations among the monomials of MZVs, Theorem~\ref{T:Main Thm} has shown that all $\bar{k}$-algebraic relations are coming from the $k$-linear relations among the same weight monomials of MZVs. However, there still remains the key problem of  finding all the $k$-linear relations among the same weight monomials of MZVs. Note that the base field $k$ plays the analogue of $\QQ$, but unlike the classical case the prime field in our setting is $\FF_{p}$, the fixed field of the Frobenius $p$-th power operator. However, more $\FF_{p}$-linear relations among MZVs are understood in \cite{Thakur10}.

\subsection{Multiple polylogarithms}
Classical multiple polylogarithms with several variables are generalizations of polylogarithms and their specializations at $(1,\ldots,1)$ give the MZVs. This phenomenon becomes delicate in the function field setting. In \cite{AT90}, Anderson-Thakur  established that Carlitz zeta value at $n\in \NN$ (ie., the multizeta value of weight $n$ and depth one) can be expressed as a $k$-linear combination of the $n$-th Carlitz polylogarithm at integral points.

In this article, we define the Carlitz multiple polylogarithms (abbreviated as CMPLs), and extend the work of Anderson-Thakur to multizeta values. Precisely, using some results of \cite{AT90, AT09} we show that each MZV is expressed explicitly as a $k$-linear combination of CMPLs at integral points (cf.~Theorem~\ref{T:CMPLsMZVs}).

From the definition, one sees that these CMPLs satisfy the stuffle relations (cf.~\S\S~\ref{sub:stuffle}). Since MZVs are $k$-linear combinations of CMPLs at integral points, to prove the analogue of Conjecture~\ref{Conj:introd} we prove that the $\bar{k}$-algebra generated by CMPLs at algebraic points forms a graded algebra and it is defined over $k$. This result is addressed as Theorem~\ref{T:Main Thm2}, which implies Theorem~\ref{T:Main Thm}.  As consequences, one further has:
\begin{enumerate}
\item[$\bullet$] Each nontrivial monomial of CMPLs at algebraic points is transcendental over $k$.

\item[$\bullet$] The ratio of two different weight nontrivial monomials of CMPLs at algebraic points is transcendental over $k$.

\item[$\bullet$] Let $Z_{1}$ and $Z_{2}$ be two nonzero values which are CMPLs at algebraic points. If $Z_{1}$ and $Z_{2}$ are of the same weight, then either $Z_{1}/Z_{2}\in k$ or $Z_{1}$ and $Z_{2}$ are algebraically independent over $k$.
\end{enumerate}

 Note that as Theorem~\ref{T:Main Thm2} implies that all the $\bar{k}$-polynomial relations among the CMPLs at algebraic points are homogenous over $k$, it is natural to ask how to describe the $k$-linear relations among the same weight monomials and we wish to tackle this problem in the future. Figuring out the problem above  would be helpful to understand the relations among MZVs.

We note that Theorem~\ref{T:Main Thm2} reveals an interesting phenomenon, which occurred previously in the celebrated theorem of Baker asserting that $\mathbb{Q}$-linear independence of logarithms of algebraic numbers implies the $\overline{\mathbb{Q}}$-linear independence. This important theorem has been generalized to the contexts of abelian logarithms by W\"ustholz (cf. \cite{Wus89a, Wus89b, BW07}) and also a function field analogue of W\"ustholz theory is developed by Yu \cite{Yu97}.

\subsection{Outline and some remarks} In \S~2, we fix our notation and state our result on multizeta values. Based on the work of \cite{AT09}, one is able to create  Frobenius difference equations for which the specialization of the solution functions gives the desired MZV. We observe that the case of nonzero values as CMPLs at algebraic points shares the same property as above. Hence we shall say that such values have the {\it{MZ}} property (see Definition~\ref{Def:DefMZ}).

In \S~3, we state a general linear independence result for the nonzero values having the {\it{MZ}} property, which is stated as Theorem~\ref{T:LinIndepPsitheta}. We give a proof of Theorem~\ref{T:LinIndepPsitheta} in \S~4, and then show in \S~5 that the nonzero values as CMPLs at algebraic points have the {\it{MZ}} property, and hence appeal to Theorem~\ref{T:LinIndepPsitheta} showing Theorem~\ref{T:Main Thm2}.

We mention that the tools of proving algebraic independence using t-motives introduced by Anderson \cite{A86} come from Papanikolas \cite{P08}, which can be regarded as a function field analogue of Grothendieck's periods conjecture. Using these tools, one has the algebraic independence results on Carlitz zeta values \cite{CY07}, Drinfeld logarithms at algebraic points \cite{CP12} etc. Although one is able to construct suitable $t$-motives so that the given multizeta values or nonzero values as CMPLs at algebraic points occur as periods of the t-motives (cf.~\cite{AT09}), to obtain the more comprehensive algebraic independence results on multizeta values or CMPLs at algebraic points via Papanikolas' theory one has to compute  the dimension of the relevant $t$-motivic Galois group. However, the dimension of such $t$-motivic Galois group relies on  the information of the periods of the $t$-motive, which is closely related to the description of the rich identities that multizeta values or CMPLs at algebraic points  satisfy. Hence it would be  difficult to compute the dimension of the Galois group in question at this moment.

The overall strategy of showing Theorem~\ref{T:LinIndepPsitheta} is to use the criterion established by Anderson-Brownawell-Papanikolas \cite{ABP04} (abbreviated as ABP-criterion). We apply the ABP-criterion to lift the given $\bar{k}$-linear relations among the special values in question to the $\bar{k}[t]$-linear relations among the solution functions. Then we analyze the coefficients (functions) as well as the solution functions to show the desired result. Finally, we emphasize that in this setting using the ABP-criterion opens a door towards the general linear independence results in question, and it enables one to avoid some difficulties occurring in the computation of the relevant Galois groups via Papanikolas' theory.

\subsection*{Acknowledgements} I am grateful to M.~Papanikolas, D.~Thakur and J.~Yu for their careful reading and many helpful comments, and to P.~Deligne for pointing out an incorrect description in an earlier version of this paper. I further thank F.~Brown, D.~Brownawell, Y.~Taguchi and J.~Zhao for many helpful conversations. Most results of this paper were worked out when I visited Hong Kong University of Science and Technology and IHES. I thank them for their hospitality and M.-S.~Xiong for his kind invitation to visit HUST. Finally, I am grateful to the referee for providing many helpful comments, which greatly improve this paper. This article is dedicated to the memory of my father.

\section{Main Result for multizeta values}
\subsection{Notation}

In this paper, we adopt the following notation.
\begin{longtable}{p{0.5truein}@{\hspace{5pt}$=$\hspace{5pt}}p{5truein}}
$\FF_q$ & the finite field with $q$ elements, for $q$ a power of a
prime number $p$. \\
$\theta$, $t$ & independent variables. \\
$A$ & $\FF_q[\theta]$, the polynomial ring in the variable $\theta$ over $\FF_q$.
\\
$A_{+}$ & set of monic polynomials in A.
\\
$k$ & $\FF_q(\theta)$, the fraction field of $A$.\\
$k_\infty$ & $\laurent{\FF_q}{1/\theta}$, the completion of $k$ with
respect to the place at infinity.\\
$\overline{k_\infty}$ & a fixed algebraic closure of $k_\infty$.\\
$\ok$ & the algebraic closure of $k$ in $\overline{k_\infty}$.\\
$\CC_\infty$ & the completion of $\overline{k_\infty}$ with respect to
the canonical extension of $\infty$.\\
$|\cdot|_{\infty}$& a fixed absolute value for the completed field $\CC_{\infty}$ so that $|\theta|_{\infty}=q$.\\
$\deg$& function assigning to $x\in k_{\infty}$ its degree in $\theta$.\\

 $\power{\CC_\infty}{t}$ & ring of formal power series in $t$ over $\CC_{\infty}$.\\

 $\TT$ & ring of power series in $\power{\CC_\infty}{t}$ that are convergent on the closed unit disc, the Tate algebra over $\CC_{\infty}$.\\
$\NN$ & set of positive integers.\\
\end{longtable}

\subsection{Multizeta values}\label{sub:MZV} Given any $s\in \NN$ and nonnegative integer $d$, we define the power sum:
\[ S_{d}(s):=\sum_{{\tiny{
                     \begin{array}{c}
                       a\in A_{+} \\
                       \deg a=d \\
                     \end{array}
}}} \frac{1}{a^{s}}. \]
In analogy with the classical multiple zeta values, Thakur \cite{Thakur04} studied the following multizeta values (which we abbreviate as MZVs): for any $r$-tuple $(s_{1},\ldots,s_{r})\in \NN^{r}$,
\[ \zeta_{A}(s_{1},\ldots,s_{r}):=\sum_{d_{1}>\cdots>d_{r}\geq 0}S_{d_{1}}(s_{1})\cdots S_{d_{r}}(s_{r})=\sum\frac{1}{a_{1}^{s_{1}}\cdots a_{r}^{s_{r}}}\in k_{\infty}  ,\]where the second sum is over $(a_{1},\ldots,a_{r})\in A_{+}^{r}$ with $\deg a_{i}$ strictly decreasing. We call this MZV having {\it{depth}} $r$ and {\it{weight}} $\sum_{i=1}^{r}s_{i}$. In the case of $r=1$, the values above are the Carlitz zeta values at positive integers. Note that each MZV is nonzero by the work of Thakur \cite{Thakur09a}. Note further that there is no natural
order on polynomials in contrast to integers, so unlike the
classical case, it is not immediately
clear that the span of MZVs is an algebra, but this together with period
interpretation was conjectured and then proved in \cite{Thakur09b, AT09, Thakur10}.

Let $Z_{1},\ldots,Z_{n}$ be MZVs of weights $w_{1},\ldots,w_{n}$ respectively. For nonnegative integers $m_{1},\ldots,m_{n}$, not all zero,  we define the (total) weight of the monomial $Z_{1}^{m_{1}}\ldots Z_{n}^{m_{n}}$ to be
\[ \sum_{i=1}^{n}m_{i} w_{i}  .\] Let $\overline{\mathcal{Z}}_{w}$ (resp. $\mathcal{Z}_{w}$) be the $\bar{k}$-vector space (resp. $k$-vector space) spanned by weight $w$ MZVs, and let $\overline{\mathcal{Z}}$ (resp. $\mathcal{Z}$) be the $\bar{k}$-algebra (resp. $k$-algebra) generated by all MZVs.  Note that by \cite{Thakur09b} we have $\mathcal{Z}_{w} \mathcal{Z}_{w'}\subseteq \mathcal{Z}_{w+w'}$. The following result is an analogue of Conjecture~\ref{Conj:introd}, and its proof is given in \S~\ref{sub:ProofMainThm}.

\begin{theorem}\label{T:Main Thm}
Let $w_{1},\ldots,w_{\ell}$ be $\ell$ distinct positive integers. Let $V_{i}$ be a finite set consisting of some monomials of multizeta values of total weight $w_{i}$ for $i=1,\ldots,\ell$. If $V_{i}$ is a linearly independent set over $k$, then the set
\[ \left\{ 1\right\}\bigcup_{i=1}^{\ell} V_{i} \] is linearly independent over $\bar{k}$. In particular, we have
\begin{enumerate}
\item $\overline{\mathcal{Z}}$ forms a graded algebra, i.e., $\overline{\mathcal{Z}}=\bar{k}\oplus_{w\in \NN}\overline{\mathcal{Z}}_{w}$.

\item $\overline{\mathcal{Z}}$ is defined over $k$ in the sense that the canonical map $\ok\otimes_{k}\mathcal{Z}\rightarrow \overline{\mathcal{Z}}$
is bijective.
\end{enumerate}
\end{theorem}

\begin{corollary}
Each nontrivial monomial of multizeta values is transcendental over $k$.
\end{corollary}

\begin{corollary}
The ratio of two different weight nontrivial monomials of multizeta values is transcendental over $k$.
\end{corollary}
\subsection{Euler dichotomy}\label{sec:EulerDichMZV} Let $\tilde{\pi}$ be a fundamental period of the Carlitz module defined in (\ref{E:tildepi}). In analogy with Euler's formula for the classical Riemann zeta function at even positive integers, Carlitz \cite{Carlitz} showed that for a positive integer $n$ divisible by $q-1$ one has
\begin{equation}\label{E:Euler-Carlitz relations}
 \zeta_{A}(n)=c_{n}\tilde{\pi}^{n},
\end{equation}where $c_{n}$ is in $k^{\times}$ and can be expressed in terms of Bernoulli-Carlitz numbers and  Carlitz factorials (cf. \cite{Goss96, Thakur04}). We shall call a positive integer $n$ \lq\lq even\rq\rq \hbox{ } if $(q-1)| n$; otherwise it is called \lq\lq odd\rq\rq. Therefore, we shall call a weight $w$ multizeta value $Z$ {\it{Eulerian}} if the ratio $Z/\tilde{\pi}^{w}$ is in $k$. Using Theorem~\ref{T:Main Thm} we have following result.

\begin{theorem}\label{T:EulerDichMZV}
Let $Z_{1},Z_{2}$ be two multizeta values of the same weight $w$. Then either the ratio $Z_{1}/Z_{2}$ is in $k$ or $Z_{1}$ and $Z_{2}$ are algebraically independent over $k$.
\end{theorem}
\begin{proof}
Suppose that $Z_{1}/Z_{2}\notin k$. Thus, by Theorem~\ref{T:Main Thm} the ratio $Z_{1}/Z_{2}$ is transcendental over $k$. If $Z_{1}$ and $Z_{2}$ are algebraically dependent over $k$, then by Theorem~\ref{T:Main Thm} there exists a homogenous polynomial $F(X,Y)\in k[X,Y]$ of positive degree so that $F(Z_{1},Z_{2})=0$. Let $d$ be the total degree of $F$. Then dividing the equation $F(Z_{1},Z_{2})=0$ by $Z_{2}^{d}$ we see that the ratio $Z_{1}/Z_{2}$ satisfies a nontrivial polynomial over $k$, whence a contradiction.
\end{proof}

Let $Z$ be a MZV of weight $w$. If the ratio $Z/\tilde{\pi}^{w}$ is algebraic over $k$, then by Corollary~\ref{Cor:DescentMZVs} we have the descent property of $Z/\tilde{\pi}^{w}$, and hence we derive the following {\it{Euler dichotomy}} phenomenon from Theorem~\ref{T:EulerDichMZV}.
\begin{corollary}\label{Cor:EulerDich}
Every multizeta value is either Eulerian or is algebraically independent from $\tilde{\pi}$. In particular, every multizeta value of \lq\lq odd \rq\rq weight $w$ is algebraically independent from $\tilde{\pi}$.
\end{corollary}
\begin{proof}
Let $Z$ be a multizeta value of weight $w$ and suppose that $Z/\tilde{\pi}^{w}\notin k$. Thus by Corollary~\ref{Cor:DescentMZVs} we have $Z/\tilde{\pi}^{w}\notin \bar{k}$. It follows from (\ref{E:Euler-Carlitz relations}) that  $Z^{q-1}/\zeta_{A}(w(q-1))\notin \bar{k}$. So Theorem~\ref{T:EulerDichMZV} implies the algebraic independence of $Z^{q-1}$ and $\zeta_{A}(w(q-1))$ over $k$, whence the algebraic independence of $Z$ and $\tilde{\pi}^{w}$ (because of (\ref{E:Euler-Carlitz relations})), which implies the algebraic independence of $Z$ and $\tilde{\pi}$.

To show the second assertion, we need to only consider $q>2$ since all positive integers are \lq\lq even\rq\rq\hbox{ } in the case of $q=2$. For $q>2$ one observes that from the definition (\ref{E:tildepi}) we have $\tilde{\pi}^{w}\notin k_{\infty}$ if $w$ is not a multiple of $q-1$. Since every MZV is in $k_{\infty}$, every MZV of \lq\lq odd\rq\rq \hbox{ }weight is not Eulerian and so the assertion follows from the previous one.
\end{proof}

\begin{remark}
Thakur \cite[Thm.~5.10.12]{Thakur04} first observed that $\zeta_{A}(2, 1)$ or $\zeta_{A}(1,2)$ is not Eulerian in the case of $q=2$ (note that MZVs and $\tilde{\pi}$  belong to $k_{\infty}$ in this case), and hence one of them is algebraically independent from Carlitz zeta values when $q=2$. In other words, there is an MZV which is algebraically independent from all Carlitz zeta values. This gives a positive answer of the analogous question in \cite[p.~231]{Andre04}.
\end{remark}

\section{Linear independence of special values occurring from difference equations}
In this section, the main goal is to establish a linear independence result of certain special values occurring from difference equations which is applied to prove Theorem~\ref{T:Main Thm}.
\subsection{Twisting operators}

For any integer $n$, we define the $n$-fold twisting on the field of Laurent series $\laurent{\CC_\infty}{t}$:
\[
     \begin{array}{ccc}
       \laurent{\CC_\infty}{t} & \rightarrow & \laurent{\CC_\infty}{t} \\
       f:=\sum a_{i}t^{i} & \mapsto & f^{(n)}:=\sum a_{i}^{q^{n}}t^{i}. \\
     \end{array}
   \  \]
We note that
\begin{equation}\label{E:f=f(-1)}
 \left\{ f\in \bar{k}(t); \hbox{ }f^{(-1)}=f  \right\}=\FF_{q}(t).
\end{equation}
Note further that the $n$-fold twisting is extended to act on $\Mat_{m\times n}(\laurent{\CC_{\infty}}{t})$ entrywise.

Throughout this paper, we fix a $(q-1)$-th root of $-\theta$ and denote it by $\tilde{\theta}$. The function
\[  \Omega(t):=\tilde{\theta}^{-q}\prod_{i=1}^{\infty} \left( 1-\frac{t}{\theta^{q^{i}}} \right) \]
has a power series expansion in $t$, and is entire on $\CC_{\infty}$ and satisfies the following difference equation:
\begin{equation}\label{E:Omega}
\Omega^{(-1)}(t)=(t-\theta) \Omega(t).
\end{equation}
Moreover, the following value
\begin{equation}\label{E:tildepi}
 \tilde{\pi}:=\frac{1}{\Omega(\theta)}
\end{equation}
is a fundamental period of the Carlitz module (cf. \cite{AT90, ABP04}).
\subsection{Anderson-Brownawell-Papanikolas criterion} We define $\mathcal{E}$ to be the ring consisting of formal power series
\[ \sum_{n=0}^{\infty}a_{n}t^{n}\in \power{\bar{k}}{t} \]
such that
\[ \lim_{n\rightarrow \infty}\sqrt[n]{|a_{n}|_{\infty}}=0,\hbox{ }[k_{\infty}\left(a_{0},a_{1},a_{2},\ldots  \right):k_{\infty}]<\infty .\]
Then any $f$ in $\mathcal{E}$ has an infinite radius of convergence with respect to $|\cdot|_{\infty}$ and has the property that $f(\alpha)\in \overline{k_{\infty}}$ for any $\alpha\in \overline{k_{\infty}}$. Any function in $\mathcal{E}$ is called an entire function and one observes that $\Omega\in \mathcal{E}$.

To state and show the main result of this section, we shall review the ABP-criterion (for abbreviation, ABP stands for Anderson-Brownawell-Papanikolas).
\begin{theorem}\textnormal{(Anderson-Brownawell-Papanikolas, \cite[Thm.~3.1.1]{ABP04})}\label{T:ABP}
Fix a matrix $\Phi\in \Mat_{\ell}(\bar{k}[t])$ so that $\det \Phi=c(t-\theta)^{s}$ for some $c\in \bar{k}^{\times}$ and some nonnegative integer $s$. Suppose that there exists a vector $\psi\in \Mat_{\ell\times 1}(\mathcal{E})$ satisfying
\[ \psi^{(-1)}=\Phi\psi.\]Then for each row vector $\rho\in \Mat_{1\times \ell}(\bar{k})$ such that $\rho \psi(\theta)=0$, there exists a vector $P\in \Mat_{1\times \ell}(\bar{k}[t])$ such that \[ P(\theta)=\rho\hbox{ and }P\psi=0 .\]
\end{theorem}

The spirit of the ABP-criterion is that every $\bar{k}$-linear relation among the entries of $\psi(\theta)$ can be lifted to a $\bar{k}[t]$-linear relation among the entries of $\psi$.
\begin{remark}
In \cite{C09}, a refined version of the ABP-criterion which relaxes the condition of $\Phi$ and the specialization of $\psi$ at more algebraic points is given. But here the ABP-criterion is sufficient for our proof.
\end{remark}

\subsection{Some notation}
Considering square matrices $M_{i}\in \Mat_{n_{i}}(\power{\CC_{\infty}}{t})$  for $i=1,\ldots,\ell$, we define $\oplus_{i=1}^{\ell}M_{i}$ to be the block diagonal matrix
\[\left(
    \begin{array}{ccc}
      M_{1} &  &  \\
       & \ddots &  \\
       &  & M_{\ell} \\
    \end{array}
  \right).
   \] For column vectors $\bv_{1},\ldots,\bv_{m}$ with entries in $\power{\CC_{\infty}}{t}$, we define $\oplus_{i=1}^{m} \bv_{i}$ to be the column vector
   \[ \left( \bv_{1}^{\rm{\tiny{tr}}}, \ldots, \bv_{m}^{\rm{\tiny{tr}}} \right)^{\rm{\tiny{tr}}}  .\]

\subsection{A linear independence result}
\begin{definition}\label{Def:DefMZ}A nonzero element $Z\in \overline{k_{\infty}}^{\times}$ is said to have the {\it{MZ}} (Multizeta) property with weight $w$  if there exists $\Phi\in \Mat_{d}(\bar{k}[t])$ and $\psi\in \Mat_{d\times 1}(\mathcal{E})$ with $d\geq 2$ so that
\begin{enumerate}
\item[$(1)$] $\psi^{(-1)}=\Phi \psi$ and $\Phi$ satisfies the conditions of the ABP-criterion;
\item[$(2)$] The last column of $\Phi$ is of the form $(0,\ldots,1)^{\tiny{\rm{tr}}}$ (whose entries are zero except the last entry being $1$);
\item[$(3)$] $\psi(\theta)$ is of the form (with specific first and last entries):
\[ \psi(\theta)=\left(
                  \begin{array}{c}
                    1/\tilde{\pi}^{w} \\
                    \vdots \\
                    cZ/\tilde{\pi}^{w} \\
                  \end{array}
                \right)
  \] for some $c\in k^{\times}$ ;
\item[$(4)$]  for any positive integer $N$, $\psi(\theta^{q^{N}})$ is of the form:
\[  \psi(\theta^{q^{N}})=\left(
                  \begin{array}{c}
                    0 \\
                    \vdots \\
                    \left(cZ/\tilde{\pi}^{w}\right)^{q^{N}}\\
                  \end{array}
                \right)\]
                (whose entries are zero except the last entry).
\end{enumerate}
\end{definition}
\begin{remark}
We will see from Theorem~\ref{T:LinIndepPsitheta} that any nonzero $Z$ having the {\it{MZ}} property has a unique weight.
\end{remark}
\begin{remark}
In \S\S~\ref{sub:CMPLs} we introduce the Carlitz multiple polylogarithm $\Li_{\mathfrak{s}}$ associated to each $r$-tuple $\mathfrak{s}=(s_{1},\ldots,s_{r})\in \NN^{r}$, and show in Proposition~\ref{P:CMPLMZ} that any nonzero value as $\Li_{\mathfrak{s}}$ at an algebraic point satisfies the MZ property with weight $s_{1}+\cdots+s_{r}$. We mention that in this situation the $d$ of Definition~\ref{Def:DefMZ} is $r+1$.
\end{remark}

\begin{proposition}\label{P:monomialsMZ}
Let $Z_{1},\ldots,Z_{n}$ be nonzero values in $\overline{k_{\infty}}^{\times}$ having the {\it{MZ}} property with weights $w_{1},\ldots,w_{n}$ respectively. For nonnegative integers $m_{1},\dots,m_{n}$, not all zero,  the monomial \[ Z_{1}^{m_{1}}\cdots Z_{n}^{m_{n}} \]
has the {\it{MZ}} property with weight $\sum_{i=1}^{n} m_{i}w_{i}$.
\end{proposition}
\begin{proof}  We consider the Kronecker product:
\[\Phi:= \Phi_{1}^{\otimes m_{1}}\otimes\cdots \otimes\Phi_{n}^{\otimes m_{n}}\hbox{ and }\psi:=  \psi_{1}^{\otimes m_{1}}\otimes\cdots \otimes\psi_{n}^{\otimes m_{n}}.  \] Then one has $\psi^{(-1)}=\Phi \psi$. Since each triple $(\Phi_{i},\psi_{i},Z_{i})$ satisfies $(1)-(4)$ of Definition~\ref{Def:DefMZ}, one sees that the triple $(\Phi,\psi,Z_{1}^{m_{1}}\cdots Z_{n}^{m_{n}})$ satisfies the conditions of Definition~\ref{Def:DefMZ} and hence $Z_{1}^{m_{1}}\cdots Z_{n}^{m_{n}}$ has the {\it{MZ}} property with weight $\sum_{i=1}^{n} m_{i}w_{i}$.

\end{proof}

 The main result in this section is stated as follows, and its proof occupies the next section.
 \begin{theorem}\label{T:LinIndepPsitheta}
Let $w_{1},\ldots,w_{\ell}$ be $\ell$ distinct positive integers. Let $V_{i}$ be a finite set of values in $\overline{k_{\infty}}^{\times}$ having the ${\it{MZ}}$-property with weight $w_{i}$, and suppose that $V_{i}$ is a linearly independent set over $k$ for $i=1,\ldots,\ell$. Then the union
\[ \left\{ 1 \right\}\bigcup_{i=1}^{\ell} V_{i} \]
is a linearly independent set over $\bar{k}$.
\end{theorem}
\section{Proof of Theorem~\ref{T:LinIndepPsitheta} and a descent property}
In this section, we give a proof of Theorem~\ref{T:LinIndepPsitheta}.  Let notation and assumptions be given in Theorem~\ref{T:LinIndepPsitheta}. Without loss of generality, we may assume that $w_{1}>\cdots> w_{\ell}$. Suppose on the contrary that the set
\[ \left\{ 1 \right\}\bigcup_{i=1}^{\ell} V_{i} \] is linearly dependent over $\bar{k}$.
 By induction on the weight, we may further assume that there are nontrivial $\bar{k}$-linear relations  connecting $V_{1}$ and $\left\{ 1 \right\}\bigcup_{i=2}^{\ell} V_{i}$. Under such hypotheses, we complete the proof in the following two steps.
 \begin{enumerate}
 \item[{\bf{Step I}}]: We show that $V_{1}$ is a linearly dependent set over $\bar{k}$;

 \item[{\bf{Step II}}]: We show that $V_{1}$ is a linearly dependent set over $k$, whence a contradiction.
 \end{enumerate}

\subsection{Proof of Step I}
In this step, our goal is to show that $V_{1}$ is a linearly dependent set over $\bar{k}$. Let $V_{i}$ consist of  $\left\{ Z_{i1},\ldots,Z_{im_{i}} \right\}$ of the same weight $w_{i}$ for $i=1,\ldots,\ell$. For $1\leq i\leq \ell$, since $Z_{ij}$ has the {\it{MZ}} property there exists $\Phi_{ij}\in \Mat_{d_{ij}}(\bar{k}[t])$ and $\psi_{ij}\in \Mat_{d_{ij}\times 1}(\mathcal{E})$ (with $d_{ij}\geq 2$) satisfying Definition~\ref{Def:DefMZ} (corresponding to the $Z_{ij}$) for $j=1,\ldots,m_{i}$.

 Define the block diagonal matrix
 \[ \tilde{\Phi}:=\oplus_{i=1}^{\ell} \left(\oplus_{j=1}^{m_{i}} (t-\theta)^{w_{1}-w_{i}} \Phi_{ij} \right)     \]
and the column vector
\[ \tilde{\psi}:=  \oplus_{i=1}^{\ell} \left(\oplus_{j=1}^{m_{i}} \Omega^{w_{1}-w_{i}} \psi_{ij} \right) .\]
Then one has $\tilde{\psi}^{(-1)}=\tilde{\Phi}\tilde{\psi}$. From Definition~\ref{Def:DefMZ}, it follows that $\tilde{\psi}(\theta)$ is of the form:
\[\tilde{\psi}(\theta)=  \oplus_{j=1}^{m_{1}} \left(
                                                         \begin{array}{c}
                                                           1/\tilde{\pi}^{w_{1}} \\
                                                           \vdots \\
                                                           (c_{1j}Z_{1j}/\tilde{\pi}^{w_{1}}) \\
                                                         \end{array}
                                                       \right)\oplus_{i=2}^{\ell} \left( \oplus_{j=1}^{m_{i}} \left(
                                                                                                                \begin{array}{c}
                                                                                                                  1/\tilde{\pi}^{w_{1}} \\
                                                                                                                  \vdots \\
                                                                                                                  c_{ij}Z_{ij}/\tilde{\pi}^{w_{1}} \\
                                                                                                                \end{array}
                                                                                                              \right) \right).\]
                                                                                                             Note that  since $w_{1}>w_{i}$ for $2\leq i\leq \ell$ and $\Omega(t)$ has simple zero at $t=\theta^{q^{N}}$ for $N\in \NN$,  for any positive integer $N$ we see that                                                                                                                $\tilde{\psi}(\theta^{q^{N}})$ is of the form
 \begin{equation}\label{E:psitilde thetaN}
 \tilde{\psi}(\theta^{q^{N}})=  \oplus_{j=1}^{m_{1}} \left(
                                                         \begin{array}{c}
                                                           0 \\
                                                           \vdots \\
                                                           0 \\
                                                           (c_{1j}Z_{1j}/\tilde{\pi}^{w_{1}})^{q^{N}} \\
                                                         \end{array}
                                                       \right)\oplus_{i=2}^{\ell} \left( \oplus_{j=1}^{m_{i}} \left(
                                                                                                                \begin{array}{c}
                                                                                                                  0 \\
                                                                                                                  \vdots \\
                                                                                                                  0 \\
                                                                                                                  0 \\
                                                                                                                \end{array}
                                                                                                              \right)
                                                        \right)
  .\end{equation}

Since by assumption $\left\{ 1 \right\}  \bigcup_{i=1}^{\ell}\left\{ Z_{i1},\ldots,Z_{im_{i}} \right\}$ is linearly dependent over $\bar{k}$, there exists a nonzero vector
$\rho$ for which $\rho\tilde{\psi}(\theta)=0$. We write \[\rho=\left( \bv_{11},\ldots,\bv_{1m_{1}},\ldots,\bv_{\ell 1},\ldots,\bv_{\ell m_{\ell}} \right),\]
where $\bv_{ij}\in \Mat_{1\times d_{ij}}(\bar{k})$ for $1\leq j\leq m_{i}$, $1\leq i\leq \ell$. Since we assume that there are non-trivial $\bar{k}$-linear relations connecting $V_{1}$ and $\left\{ 1 \right\}\bigcup_{i=2}^{\ell}V_{i}$, the last entry of $\bv_{1s}$ is nonzero for some $1\leq s\leq m_{1}$.

By Theorem~\ref{T:ABP} for each $1\leq i\leq \ell$ there exists $\bff_{ij}\in \Mat_{1\times d_{ij}}(\bar{k}[t])$ (for $j=1,\ldots,m_{i}$) so that ${\mathbf{F}}:=\left( \bff_{11},\ldots,\bff_{1m_{1}},\ldots,\bff_{\ell 1},\ldots,\bff_{\ell m_{\ell}} \right)$ satisfies
\[{\mathbf{F}}\tilde{\psi}=0\hbox{ and }{\mathbf{F}}(\theta)=\rho .\]
  Since by hypothesis the last entry of $\bv_{1s}$ is nonzero, the last entry of $\bff_{1s}$ is a nontrivial polynomial. We pick an integer $N$ sufficiently large for which the last entry of $\bff_{1s}$ is non-vanishing at $t=\theta^{q^{N}}$. Specializing the equation ${\mathbf{F}}\tilde{\psi}=0$ at $t=\theta^{q^{N}}$ and using (\ref{E:psitilde thetaN}) gives rise to a nontrivial $\bar{k}$-linear relation among  \[ Z_{11}^{q^{N}},\ldots,Z_{1m_{1}}^{q^{N}}  .\]
Since our field is of characteristic $p$, by taking the $q^{N}$-th root from the $\bar{k}$-linear relation above we obtain a nontrivial $\bar{k}$-linear relation among the weight $w_{1}$ values  $\left\{ Z_{11},\ldots,Z_{1m_{1}} \right\}$, as claimed.

\subsection{Proof of Step II}
In this step, our goal is to show that $V_{1}$ is a linearly dependent set over $k$, whence a contradiction and thus we complete the proof of Theorem~\ref{T:LinIndepPsitheta}. According to Step I above, we have shown that $V_{1}$ is linearly dependent over $\bar{k}$. Without confusion with the notation of double index in Step I, for simplicity we write $V_{1}=\left\{ Z_{1},\ldots,Z_{m} \right\}$, and without loss of generality we may assume that $m\geq 2$ and
\[ \dim_{\bar{k}} \bar{k}\textnormal{-Span}\left\{ V_{1}\right\}=m-1 .\]  Again for simplicity and without confusion with the double index above, we let $\Phi_{j}\in \Mat_{d_{j}}(\bar{k}[t])$ and $\psi_{j}\in \Mat_{d_{j}\times 1}(\mathcal{E})$ (with $d_{j}\geq 2$) be associated to the value $Z_{j}$ having the {\it{MZ}} property with  weight $w_{1}$ for $j=1,\ldots,m$.

Define the block diagonal matrix  \[\Phi:=\oplus_{j=1}^{m}\Phi_{j} \]
and define the column vector
\[ \psi:=\oplus_{j=1}^{m}\psi_{j}  .\]Notice that
\begin{equation}\label{E:psitheta}
 \psi(\theta)=\oplus_{j=1}^{m} \left(
                                                         \begin{array}{c}
                                                           1/\tilde{\pi}^{w_{1}} \\
                                                           \vdots \\
                                                           c_{j}Z_{j}/\tilde{\pi}^{w_{1}}\\
                                                         \end{array}
                                                       \right) \end{equation}for some $c_{j}\in k^{\times}$, and for $N\in \NN$ we have
\begin{equation}\label{E:psitheta qN}
   \psi(\theta^{q^{N}})=\oplus_{j=1}^{m} \left(
                                                         \begin{array}{c}
                                                          0\\
                                                           \vdots \\
                                                           (c_{j}Z_{j}/\tilde{\pi}^{w_{1}})^{q^{N}}\\
                                                         \end{array}
                                                       \right).
\end{equation}

 Without loss of generality, we may assume that $Z_{1}\in \bar{k}\textnormal{-Span}\left\{ Z_{2},\ldots,Z_{m} \right\}$, and so by hypothesis $\left\{ Z_{2},\ldots,Z_{m} \right\}$ is linearly independent over $\bar{k}$. By the ABP-criterion (Theorem~\ref{T:ABP}) there exists vectors
 $\bff_{j}=\left( \bff_{j1},\ldots,\bff_{jd_{j}} \right)\in \Mat_{1\times d_{j}}(\bar{k}[t])$ for $j=1,\ldots,m$ so that if we put $\bF:=\left( \bff_{1},\ldots,\bff_{m} \right)$ then we have
\begin{equation}
\bF\psi=0, \hbox{ }\bff_{1d_{1}}(\theta)=1 \hbox{ and }\bff_{jh}(\theta)=0\hbox{ for all }1\leq h< d_{j}.
\end{equation}
For convenience we shall refer \lq\lq(i,j)-component of $\bF$\rq\rq to the entry ${\bf{f}}_{ij}$ with double index $ij$.

We divide the vector $\bF$ by $\bff_{1d_{1}}$, and write $\bG:=\frac{1}{\bff_{1d_{1}}}\bF$. Let $d:=\sum_{j=1}^{m}d_{j}$. Note that the vector $\bG$ is of the form
\[ \bG=\left( \bg_{11},\ldots,1,\ldots,\bg_{m1},\ldots,\bg_{md_{m}}  \right)\in \Mat_{1\times d}(\bar{k}(t)), \]
where  $1$ is corresponding to the $(1,d_{1})$-component of $\bG$,
and we have
\begin{equation} \label{E:vanishing of Gij theta}
\bG\psi=0\hbox{ and }\bg_{jh}(\theta)=0\hbox{ for all }1\leq h<d_{j}.
\end{equation}

We use the $(-1)$-fold twisting action on $\bG \psi=0$, and so obtain $\bG^{(-1)}\Phi \psi=0$. Subtracting this equation from $\bG\psi=0$ we obtain that
\begin{equation}\label{E:P-P-1Phi}
\left(\bG- \bG^{(-1)}\Phi \right)\psi=0   .
\end{equation}
Note that the last column of each matrix $\Phi_{j}$ is $(0,\ldots,0,1)^{\tiny{\hbox{tr}}}$, and hence the $(1,d_{1})$-component of $\bG- \bG^{(-1)}\Phi $ is zero since the $(1,d_{1})$-component of the vector $\bG$ is $1$. We further note that the $(1,\sum_{i=1}^{j}d_{i})$-entry of $\bG- \bG^{(-1)}\Phi$ is equal to
\[\bg_{jd_{j}}-\bg_{jd_{j}}^{(-1)} \hbox{ for }j=2,\ldots,m.\]
We claim that $\bg_{jd_{j}}-\bg_{jd_{j}}^{(-1)}=0 \hbox{ for }j=2,\ldots,m$.

To prove the claim above, suppose on the contrary that there exists some $2\leq j\leq m$ for which $\bg_{jd_{j}}-\bg_{jd_{j}}^{(-1)}$ is nonzero. We pick an $N\in \NN$ sufficiently large for which all entries of $\left(\bG- \bG^{(-1)}\Phi \right)$ are regular at $t=\theta^{q^{N}}$, and $\bg_{jd_{j}}-\bg_{jd_{j}}^{(-1)}$ is non-vanishing at $t=\theta^{q^{N}}$. Specializing (\ref{E:P-P-1Phi}) at $t=\theta^{q^{N}}$ and using (\ref{E:psitheta qN}) we obtain a nontrivial $\bar{k}$-linear relations among $Z_{2}^{q^{N}},\ldots,Z_{m}^{q^{N}}$ because the $(1,d_{1})$-component of $\bG- \bG^{(-1)}\Phi $ is zero. By taking a $q^{N}$-th root we obtain a nontrivial $\bar{k}$-linear relation among $Z_{2},\ldots,Z_{m}$, whence a contradiction since we assume that $Z_{2},\ldots,Z_{m}$ are linearly independent over $\bar{k}$.

Thus by (\ref{E:f=f(-1)}) we have that $\bg_{jd_{j}}\in \FF_{q}(t)$ for  $j=2,\ldots,m$. Note that each entry of $\bG$ is regular at $t=\theta$. By specializing the equation $\bG\psi=0$ at $t=\theta$ and using (\ref{E:psitheta}) and (\ref{E:vanishing of Gij theta}), we obtain a nontrivial $k$-linear relation among $Z_{1},\ldots,Z_{m}$. This contradicts to our assumption, and hence we finish the proof.

\subsection{A descent property} To obtain the Eulerian dichotomy phenomenon (cf.~Corollary~\ref{Cor:EulerDich}) we need to establish the following descent property.
\begin{proposition}\label{P:DescentMZ}
Let $Z_{1},\ldots,Z_{n}$ be nonzero values in $\overline{k_{\infty}}^{\times}$ having the MZ property with the same weight $w$. If $\tilde{\pi}^{w},Z_{1},\ldots,Z_{n}$ are linearly dependent over $\ok$, then they are linearly dependent over $k$.
\end{proposition}
\begin{proof}
Without loss of generality we may assume that $Z_{1},\ldots,Z_{n}$ are linearly independent over $\ok$. Let $\Phi_{i}\in \Mat_{d_{i}}(\ok[t])$ and $\psi_{i}\in \Mat_{d_{i}\times 1}(\mathcal{E})$ be given in Definition~\ref{Def:DefMZ} associated to $Z_{i}$ for $i=1,\ldots,n$. Let $m=1+\sum_{i=1}^{n}d_{i}$ and define
\[ \Phi:=(1)\oplus_{i=1}^{n}\Phi_{i}\in \Mat_{m}(\ok[t]) \hbox{ and }  \psi:=(1)\oplus_{i=1}^{n}\psi_{i}\in \Mat_{m\times 1}(\mathcal{E}).\]
Then we see that $\psi^{(-1)}=\Phi\psi$ and $(\Phi,\psi)$ satisfies the conditions of Theorem~\ref{T:ABP}.

Let $b\tilde{\pi}^{w}+\sum_{i=1}^{n}a_{i}Z_{i}=0$ for some $b,a_{1},\ldots,a_{n}\in\ok$ with $b\neq 0$ and $a_{n}\neq 0$. For each $1\leq i\leq n$, the last coordinate of $\psi_{i}(\theta)$ is given by $c_{i}Z_{i}/\tilde{\pi}^{w}$ for some $c_{i}\in k^{\times}$. By Definition~\ref{Def:DefMZ}~(3) and  Theorem~\ref{T:ABP} there exist $f\in \ok[t]$ and ${\bf{f}}_{i}=(f_{i1},\ldots,f_{id_{i}})\in \Mat_{1\times d_{i}}(\ok[t])$ so that for each $1\leq i\leq n$,
\[ f(\theta)=b, f_{id_{i}}(\theta)=a_{i}/c_{i} \hbox{ and } f_{ij}(\theta)=0 \hbox{ for }  j=1,\ldots,d_{i}-1,\]
and $P\psi=0$, where $P:=\left(f,{\bf{f}}_{1},\ldots,{\bf{f}}_{n}  \right)\in \Mat_{1\times m}(\ok[t])$.

Let $g=f_{nd_{n}}$ be the last entry of the row vector of ${\bf{f}}_{n}\in\Mat_{1\times d_{n}}(\ok[t])$ and put $\widetilde{P}:=\frac{1}{g}P$. We note that $f(\theta)\neq 0$ and $g(\theta)\neq 0$ and that the last entry of $\widetilde{P}$ is $1$. Using the $(-1)$-twisting operation on the equation $\widetilde{P}\psi=0$ and then subtracting it from $\widetilde{P}\psi=0$ we obtain that
\begin{equation}\label{E:P-PPhipsi}
\left(\widetilde{P}-\widetilde{P}^{(-1)} \Phi\right)\psi=0.
\end{equation}

Note that the last entry of $\widetilde{P}-\widetilde{P}^{(-1)} \Phi$ is zero because of Definition~\ref{Def:DefMZ}~(2). Now we pick a sufficiently large integer $N$ so the all the entries of $\widetilde{P}$ are regular at $t=\theta^{q^{N}}$. By using Definition~\ref{Def:DefMZ}~(4) and specializing (\ref{E:P-PPhipsi}) at $t=\theta^{q^{N}}$ we derive a $\ok$-linear relation between $\tilde{\pi}^{wq^{N}},Z_{1}^{q^{N}},\ldots,Z_{n-1}^{q^{N}}$, whence obtaining a $\ok$-linear relation among $\tilde{\pi}^{w},Z_{1},\ldots,Z_{n-1}$ after taking the $q^{N}$-th root of the equation. From the hypothesis on the $\ok$-linear independence of $\tilde{\pi}^{w},Z_{1},\ldots,Z_{n-1}$ the coefficients of the $\ok$-linear equation obtained above have to be zero, particularly for $i=1,\ldots,n-1$,
\[\left(f/g-(f/g)^{(-1)} \right)(\theta^{q^{N}})=\left(f_{id_{i}}/g-(f_{id_{i}}/g)^{(-1)} \right)(\theta^{q^{N}})= 0 \hbox{ for }N\gg0.\] It follows from (\ref{E:f=f(-1)}) that $f/g,f_{id_{i}}/g\in \FF_{q}(t)$ for $i=1,\ldots,n-1$. By specializing the equation $\widetilde{P}\psi=0$ at $t=\theta$ we obtain the desired result.

\end{proof}

\section{Linear independence of monomials of Carlitz multiple polylogarithms}
In \cite{AT90}, Anderson and Thakur showed that the Carlitz zeta value at $n\in \NN$ can be expressed as a  $k$-linear combination of the $n$-th Carlitz polylogarithm at integral points in $A$ (cf.~\cite[\S\S~3.9]{AT90}). In this section, we first define the Carlitz multiple polylogarithms (abbreviated as CMPLs) and extend the work of Anderson-Thakur to multizeta values. We then show that the nonzero values which are CMPLs at algebraic points satisfy the {\it{MZ}} property and hence using Theorem~\ref{T:LinIndepPsitheta} we derive Theorem~\ref{T:Main Thm2}.

\subsection{Carlitz multiple polylogarithms}\label{sub:CMPLs}
We define $\mathcal{L}_{0}:=1$ and $\mathcal{L}_{i}:=\prod_{j=1}^{i}(\theta-\theta^{q^{j}})$ for $i\in \NN$. For $n\in \NN$, the $n$-th Carlitz polylogarithm is defined by
\[\log_{n}(z):=\sum_{i=0}^{\infty}\frac{z^{q^{i}}}{\mathcal{L}_{i}^{n}}.\] (Note that in \cite{AT90, Goss96} it is called the $n$-th Carlitz multilogarithm). It converges on the disc $\left\{ z\in \CC_{\infty};\hbox{ }|z|_{\infty}<q^{\frac{nq}{q-1}} \right\}$.

\begin{definition}
Given any  $\mathfrak{s}=(s_{1},\ldots,s_{r})\in \NN^{r}$, we define its associated Carlitz multiple polylogarithm as the following:
\[ {\rm{Li}}_{\mathfrak{s}}(z_{1},\ldots,z_{r}):=\sum_{i_{1}>\cdots>i_{r}\geq 0} \frac{z_{1}^{q^{i_{1}}}\cdots z_{r}^{q^{i_{r}}}  }{\mathcal{L}_{i_{1}}^{s_{1}}\cdots \mathcal{L}_{i_{r}}^{s_{r}}   }  .\]
\end{definition}

Note that for any nonnegative integer $i$ and positive integer $n$, we have
\[ |\mathcal{L}_{i}^{n}|_{\infty}=q^{\frac{nq(q^{i}-1)}{q-1}}  .\]
So the absolute value of the general term in the series ${\rm{Li}}_{\mathfrak{s}}(z_{1},\ldots,z_{r})$ is given by
\begin{equation}\label{E:GeneralTermCMPL}
 q^{\frac{q}{q-1}(s_{1}+\cdots+s_{r})}|z_{1}/(\theta^{\frac{qs_{1}}{q-1}}  ) |_{\infty}^{q^{i_{1}}}\cdots |z_{r}/ (\theta^{\frac{qs_{r}}{q-1}}) |_{\infty}^{q^{i_{r}}}      .
\end{equation}

\begin{definition} Given $\mathfrak{s}=(s_{1},\ldots,s_{r})\in \NN^{r}$,
we denote by
\[\DD_{\mathfrak{s}}:=\left\{\bu=(u_{1},\ldots,u_{r})\in \CC_{\infty}^{r}| \Li_{\mathfrak{s}}(\bu) \hbox{ converges } \right\},\]
the convergence domain of $\Li_{\mathfrak{s}}$. Note that by non-archimedean analysis $\DD_{\mathfrak{s}}$ is described as
\[\DD_{\mathfrak{s}}=\left\{ (u_{1},\ldots,u_{r})\in \CC_{\infty}^{r}| |u_{1}/(\theta^{\frac{qs_{1}}{q-1}}  ) |_{\infty}^{q^{i_{1}}}\cdots |u_{r}/ (\theta^{\frac{qs_{r}}{q-1}}) |_{\infty}^{q^{i_{r}}}\rightarrow 0 \hbox{ as }0\leq  i_{r}< \cdots< i_{1}\rightarrow \infty   \right\} . \]
\end{definition}

\begin{remark}
It might be complicated to explicitly describe $\DD_{\mathfrak{s}}$ in terms of $s_{1},\ldots,s_{r}$, but by (\ref{E:GeneralTermCMPL}) it is clear that $\Li_{\mathfrak{s}}$ converges on this smaller polydisc:
\[\DD_{\mathfrak{s}}':= \left\{ (u_{1},\ldots,u_{r})\in \CC_{\infty}^{r}; \hbox{ }|u_{i}|_{\infty}< q^{\frac{s_{i}q}{q-1}} \hbox{ for }i=1,\ldots,r \right\}   .\]
\end{remark}

\begin{remark}
For any $\bu=(u_{1},\ldots,u_{r})\in \DD_{\mathfrak{s}}'\cap (\CC_{\infty}^{\times})^{r}$, using (\ref{E:GeneralTermCMPL}) the general term has a unique maximal absolute value when $(i_{1},\ldots,i_{r})=(r-1,\ldots,0)$. It follows that $\Li_{\mathfrak{s}}(\bu)$ is non-vanishing. But the author does not know whether or not $\Li_{\mathfrak{s}}(\bu)$ is non-vanishing for any $\bu$ in the convergence domain $\DD_{\mathfrak{s}}$.
\end{remark}

\subsection{Stuffle relations}\label{sub:stuffle}  Note that since the indexes of the series $\Li_{\mathfrak{s}}$ are in the total ordered set $\ZZ_{\geq 0}$, the classical stuffle relations for multiple polylogarithms work here, whence deducing some natural algebraic relations among the CMPLs at algebraic points (in $\DD_{\mathfrak{s}}\cap (\ok^{\times})^{r}$).  We explain more details as the following.

Given $\mathfrak{s}=(s_{1},\ldots,s_{r})\in \NN^{r}$ and $\mathfrak{s}'=(s_{1}',\ldots,s_{r'}')\in \NN^{r'}$, fix a positive integer $r''$ with ${\rm{max}}\left\{ r,r' \right\}\leq r''\leq r+r'$. We consider a pair consisting of two vectors $\bv, \bv'\in  \ZZ_{\geq 0}^{r''}$ which are required to satisfy $\bv+\bv'\in \NN^{r''}$ and which are obtained from the following ways. One vector $\bv$ is obtained  from $\mathfrak{s}$ by inserting $(r''-r)$ zeros in all possible ways (including in front and at end), and another vector $\bv'$ is obtained from $\mathfrak{s}'$ by inserting $(r''-r')$ zeros in all possible ways (including in front and at end).

 One observes from the definition of the series that the CMPLs satisfy the stuffle relations which are analogous to the classical case (cf. \cite{W02}):
\begin{equation}\label{E:stuffle}
  {\rm{Li}}_{\mathfrak{s}}({\bf{z}}){\rm{Li}}_{\mathfrak{s}'}({\bf{z}}')=\sum_{(\bv,\bv')}{\rm{Li}}_{\bv+\bv'}({\bf{z}}''),
\end{equation}where the pair $(\bv,\bv')$ runs over all the possible expressions as above for all $r''$ with ${\rm{max}}\left\{ r,r' \right\}\leq r''\leq r+r'$. For each such $\bv+\bv'\in \NN^{r''}$, the component $z_{i}''$ of ${\bf{z}''}$ is $z_{j}$ if the $i$th component of $\bv$ is $s_{j}$ and the $i$th component of $\bv'$ is $0$, it is $z_{\ell}'$ if the $i$th component of $\bv$ is $0$ and the $i$th component of $\bv'$ is $s_{\ell}'$, and finally it is $z_{j}z_{\ell}'$ if the $i$th component of $\bv$ is $s_{j}$ and the $i$th component of $\bv'$ is $s_{\ell}'$.

For example, for $r=r'=1$ (\ref{E:stuffle}) yields
\[ {\rm{Li}}_{s}(z){\rm{Li}}_{s'}(z')={\rm{Li}}_{(s,s')}(z,z')+ {\rm{Li}}_{(s',s)}(z',z)+{\rm{Li}}_{s+s'}(zz'). \]
 For $r=1,r'=2$, one has
\[
    \begin{array}{rl}
     {\rm{Li}}_{s}(z){\rm{Li}}_{(s_{1}',s_{2}')}(z_{1}',z_{2}')  &={\rm{Li}}_{(s,s_{1}',s_{2}')}(z,z_{1}',z_{2}')+ {\rm{Li}}_{(s_{1}',s,s_{2}')}(z_{1}',z,z_{2}')+{\rm{Li}}_{(s_{1}',s_{2}',s)}(z_{1}',z_{2}',z) \\
       & +{\rm{Li}}_{(s+s_{1}',s_{2}')}(z z_{1}',z_{2}')+ {\rm{Li}}_{(s_{1}',s+s_{2}')}(z_{1}',zz_{2}').\\
    \end{array}
  \]

\subsection{Special series and formulas}  Given a polynomial $Q:=\sum_{i}a_{i}t^{i}\in \bar{k}[t]$, we define $\|Q\|_{\infty}:={\rm{max}}_{i}\left\{|a_{i}|_{\infty} \right\}$. In what follows, we consider some specific series which are generalizations of the series associated to MZVs studied in \cite[\S\S~2.5]{AT09}.
\begin{lemma}\label{L:convergence}
Given a $d$-tuple $\mathfrak{s}=(s_{1},\ldots,s_{d})\in \NN^{d}$, let $\mathfrak{Q}:=(Q_{1},\ldots,Q_{d})\in \bar{k}[t]^{d}$ satisfy that as $0\leq i_{d}< \cdots< i_{1}\rightarrow \infty$,
\[\left(\parallel Q_{1}\parallel_{\infty}/|(\theta^{\frac{qs_{1}}{q-1}}  ) |_{\infty}\right)^{q^{i_{1}}}\cdots \left(\parallel Q_{d}\parallel_{\infty}/ |(\theta^{\frac{qs_{d}}{q-1}}) |_{\infty}\right)^{q^{i_{d}}}\rightarrow 0. \]
 We define the following series
\begin{equation}\label{E:LsQ}
    \begin{array}{rl}
     L_{\mathfrak{s},\mathfrak{Q}}(t) := & \sum_{i_{1}>\cdots>i_{d}\geq 0} \left(\Omega^{s_{d}}Q_{d} \right)^{(i_{d})}\cdots \left(\Omega^{s_{1}}Q_{1} \right)^{(i_{1})}\\\
       =& \Omega^{s_{1}+\cdots+s_{d}}\sum_{i_{1}>\cdots>i_{d}\geq 0} \frac{{Q_{d}}^{(i_{d})}(t)\cdots Q_{1}^{(i_{1})}(t) }{ \left( (t-\theta^{q})\ldots(t-\theta^{q^{i_{d}}})  \right)^{s_{d}}\ldots \left((t-\theta^{q})\ldots(t-\theta^{q^{i_{1}}})  \right)^{s_{1}}} \\
    \end{array}
 \end{equation}
associated to the two $d$-tuples $\mathfrak{s}$ and $\mathfrak{Q}$. Then $L_{\mathfrak{s},\mathfrak{Q}}$  is an entire function.
\end{lemma}

\begin{proof}
Note that
\[
  \begin{array}{rl}
     & \frac{{\parallel{Q_{d}}^{(i_{d})}(t)\cdots Q_{1}^{(i_{1})}(t)\parallel_{\infty}}} { {\parallel\left( (t-\theta^{q})\ldots(t-\theta^{q^{i_{d}}})  \right)^{s_{d}}\ldots \left((t-\theta^{q})\ldots(t-\theta^{q^{i_{1}}})  \right)^{s_{1}}\parallel_{\infty}}}
 \\
    = &
q^{\frac{q}{q-1}(s_{1}+\cdots+s_{d})}\left(\parallel Q_{1}\parallel_{\infty}/|(\theta^{\frac{qs_{1}}{q-1}}  ) |_{\infty}\right)^{q^{i_{1}}}\cdots \left(\parallel Q_{d}\parallel_{\infty}/ |(\theta^{\frac{qs_{d}}{q-1}}) |_{\infty}\right)^{q^{i_{d}}} . \\
  \end{array}
 \] So the hypothesis of $\mathfrak{Q}$ implies that the series $L_{\mathfrak{s},\mathfrak{Q}}(t)$ is in $\TT$.  We claim that we can create a matrix $\Phi\in \Mat_{d+1}(\ok[t])$ and solve the system of difference equations $\psi^{(-1)}=\Phi\psi$ for $\psi\in\Mat_{(d+1)\times 1}(\power{\CC_{\infty}}{t})$ so that
\begin{enumerate}
\item[$\bullet$] $\det \Phi|_{t=0}\neq 0$ and all the entries of $\psi$ are in the Tate algebra $\TT$.
\item[$\bullet$] The last coordinate of $\psi$ is $L_{\mathfrak{s},\mathfrak{Q}}$.
\end{enumerate}
Note that the first property enables us to apply \cite[Prop.~3.1.1]{ABP04}. It follows that all the entries are actually in $\mathcal{E}$, and so is $L_{\mathfrak{s},\mathfrak{Q}}$ by the second property above.

To prove the claim above, we define
\begin{equation}\label{E:Phi}
\Phi:=\begin{pmatrix}
                (t-\theta)^{s_{1}+\cdots+s_{d}}  & 0 & 0 &\cdots  & 0 \\
                Q_{1}^{(-1)}(t-\theta)^{s_{1}+\cdots+s_{d}}  & (t-\theta)^{s_{2}+\cdots+s_{d}} & 0 & \cdots & 0 \\
                 0 &Q_{2}^{(-1)} (t-\theta)^{s_{2}+\cdots+s_{d}} &  \ddots&  &\vdots  \\
                 \vdots &  & \ddots & (t-\theta)^{s_{d}} & 0 \\
                 0 & \cdots & 0 & Q_{d}^{(-1)}(t-\theta)^{s_{d}} & 1 \\
               \end{pmatrix}
             \in \Mat_{d+1}(\ok[t]),
\end{equation}
and define the diagonal matrix
\[
\Lambda:=\left(
            \begin{array}{ccccc}
              \Omega^{s_{1}+\cdots+s_{d}} &  &  &  &  \\
               & \Omega^{s_{2}+\cdots+s_{d}} &  &  &  \\
               &  & \ddots &  &  \\
               &  &  & \Omega^{s_{d}} &  \\
               &  &  &  & 1 \\
            \end{array}
          \right)\in \Mat_{d+1}(\mathcal{E}).
\]For each $1\leq j\leq d$, we consider the two $j$-tuples $(s_{1},\ldots,s_{j})$ and $(Q_{1},\ldots,Q_{j})$ and define $L_{j+1}$ to be the series (\ref{E:LsQ}) associated to these two tuples. Then we put
 \[ L:=\left(
               \begin{array}{c}
                 1 \\
                 L_{2} \\
                 \vdots \\
                 L_{d+1} \\
               \end{array}
             \right)\in \Mat_{(d+1)\times 1}(\TT).
          \]and
          \begin{equation}\label{E:psi}
          \psi:=\Lambda L\in \Mat_{(d+1)\times 1}(\TT).
          \end{equation}
          Using the functional equation $\Omega^{(-1)}=(t-\theta)\Omega$ we see that  $\Phi$ and $\psi$ satisfy the desired claim (cf.~\cite[\S\S~2.5]{AT09}).
\end{proof}

The following lemma is the key formula so that Theorem~\ref{T:LinIndepPsitheta} applies to CMPLs at algebraic points.
\begin{lemma}\label{L:Lj+1 theta qN}
Given any $r$-tuple $\mathfrak{s}=(s_{1},\ldots,s_{r})\in \NN^{r}$, let $\bu=(u_{1},\ldots,u_{r})\in \DD_{\mathfrak{s}}\cap (\ok^{\times})^{r}$. For each $1\leq j\leq r,$ we let $L_{j+1}$ be the series defined in (\ref{E:LsQ}) associated to the two tuples $(s_{1},\ldots,s_{j})$ and $(u_{1},\ldots,u_{j})$. Then for each $1\leq j\leq r$, we have that
\[ L_{j+1}(\theta^{q^{N}})=\left( \Li_{(s_{1},\ldots,s_{j})}(u_{1},\ldots,u_{j})/\tilde{\pi}^{s_{1}+\cdots+s_{j}} \right)^{q^{N}}  \]
for any nonnegative integer $N$.
\end{lemma}

\begin{proof}
For the case $N=0$, the result follows from the second expression of $L_{j+1}$ in (\ref{E:LsQ}) and the definition of $\Li_{(s_{1},\ldots,s_{j})}(u_{1},\ldots,u_{j})$. Now, let $N$ be a positive integer.  Fixing $1\leq j\leq r$, we write $L_{j+1}=L_{j+1}^{< N}+L_{j+1}^{\geq N}$, where
\[
     \begin{array}{c}
        L_{j+1}^{< N}(t):= \sum_{{\tiny{
                                   \begin{array}{c}
                                     i_{1}>\cdots>  i_{j}\geq 0; \\
                                     i_{j}<N \\
                                   \end{array}
                                 }}
        }  \left(\Omega^{s_{j}}u_{j} \right) ^{(i_{j})}\ldots \left(  \Omega^{s_{1}}u_{1} \right) ^{(i_{1})} \\
        L_{j+1}^{\geq N}(t):= \sum_{i_{1}>\cdots>  i_{j}\geq N} \left(  \Omega^{s_{j}}u_{j} \right) ^{(i_{j})}\ldots \left(  \Omega^{s_{1}}u_{1} \right) ^{(i_{1})} .\\
     \end{array}
  \]

Now we express $L_{r+1}^{< N}(t)$  as
\[L_{j+1}^{< N}(t)=\sum_{{\tiny{
                                   \begin{array}{c}
                                     i_{1}>\cdots>  i_{j}\geq 0; \\
                                     i_{j}<N \\
                                   \end{array}
                                 }}
        }  \frac{ \Omega^{s_{1}+\cdots+s_{j}} u_{j}^{q^{i_{j}}}\ldots u_{1}^{q^{i_{1}}}  }{ \left( (t-\theta^{q})\ldots(t-\theta^{q^{i_{j}}})  \right)^{s_{j}}\ldots \left((t-\theta^{q})\ldots(t-\theta^{q^{i_{1}}})  \right)^{s_{1}}    }.    \]

We claim that $L_{j+1}^{< N}(\theta^{q^{N}})=0$. To prove this claim, we first note that the order of vanishing of $\Omega^{s_1+\cdots+s_{j}}$ at $t=\theta^{q^{N}}$ is equal to $s_{1}+\cdots+s_{j}$. On the other hand, we observe that each term in the expression of $L_{j+1}^{< N}(t)$ above may have pole at $t=\theta^{q^{N}}$ of order at most $s_{1}+\cdots+s_{j-1}$ since $i_{j}<N$. It follows that each term in the expression of $L_{j+1}^{< N}(t)$ above has positive order of vanishing at $t=\theta^{q^{N}}$, whence the claim.

Therefore, we have that $L_{j+1}(\theta^{q^{N}})=L_{j+1}^{\geq N}(\theta^{q^{N}})$ (which we will see from the following that the series $L_{j+1}^{\geq N}$ converges at $t=\theta^{q^{N}}$). By definition, we express $L_{j+1}^{\geq N}$ as
\[  L_{j+1}^{\geq N}(t)= \left(\sum_{i_{1}>\cdots>  i_{j}\geq 0} \left(  \Omega^{s_{j}}u_{j} \right) ^{(i_{j})}\ldots \left(  \Omega^{s_{1}}u_{1} \right) ^{(i_{1})} \right)^{(N)},\]
and hence
\[   L_{j+1}^{\geq N}(\theta^{q^{N}})= \left(\sum_{i_{1}>\cdots>  i_{j}\geq 0} \left(  \Omega^{s_{j}}u_{j} \right) ^{(i_{j})}\ldots \left(  \Omega^{s_{1}}u_{1} \right) ^{(i_{1})}  |_{t=\theta}\right)^{q^{N}}=\left( \frac{\Li_{(s_{1},\ldots,s_{j})}(u_{1},\ldots,u_{j})}{\tilde{\pi}^{s_{1}+\cdots+s_{j}}}  \right)^{q^{N}}.\]
\end{proof}

\subsection{Linear independence result}

The following proposition establishes that a nonzero value which is a specialization of CMPL at an algebraic point satisfies the {\it{MZ}} property.
\begin{proposition}\label{P:CMPLMZ}
Given any $r$-tuple $\mathfrak{s}=(s_{1},\ldots,s_{r})\in \NN^{r}$, we let $\bu=(u_{1},\ldots,u_{r})\in\DD_{\mathfrak{s}}\cap (\bar{k}^{\times})^{r}$. If $\Li_{\mathfrak{s}}(\bu)$ is nonzero,  then $\Li_{\mathfrak{s}}(\bu)$ has the MZ property with weight $\sum_{i=1}^{r}s_{i}$.
\end{proposition}
\begin{proof}
Let $\mathfrak{Q}:=(u_{1},\ldots,u_{r})$ and consider the series $L_{\mathfrak{s},\mathfrak{Q}}(t)$ be defined as (\ref{E:LsQ}). Let $\Phi$ be defined as (\ref{E:Phi}) and $\psi$ be defined as (\ref{E:psi}). Notice that by Lemma~\ref{L:convergence}
\[ \psi=\left(
          \begin{array}{c}
            \Omega^{s_{1}+\cdots+s_{r}} \\
            \Omega^{s_{2}+\cdots+s_{r}}L_{2} \\
            \vdots \\
            \Omega^{s_{r}}L_{r} \\
            L_{r+1}\\
          \end{array}
        \right)\in \Mat_{(r+1)\times 1}(\mathcal{E}),
  \] where $L_{j+1}$ is the series (\ref{E:LsQ}) associated to the two tuples $(s_{1},\ldots,s_{j})$ and $(u_{1},\ldots,u_{j})$. By the constructions of $\Phi$ and $\psi$ the first two properties of Definition~\ref{Def:DefMZ} are satisfied.

  By Lemma~\ref{L:Lj+1 theta qN} we see that the third property of Definition~\ref{Def:DefMZ} is satisfied. Note that $\Omega$ has a simple zero at $t=\theta^{q^{N}}$ for each $N\in \NN$ and so Lemma~\ref{L:Lj+1 theta qN} implies that $\psi(\theta^{q^{N}})$ satisfies the last property of Definition~\ref{Def:DefMZ}, whence completing the proof.

\end{proof}

\begin{definition}
Given any $r$-tuple $\mathfrak{s}=(s_{1},\ldots,s_{r})\in \NN^{r}$, let $Z$ be a nonzero value as a specialization of $\Li_{\mathfrak{s}}$ at some algebraic point in $\DD_{\mathfrak{s}}\cap (\ok^{\times})^{r}$. We define the weight of $Z$ to be ${\rm{wt}}(Z):=s_{1}+\cdots+s_{r}$. Let $Z_{1},\ldots,Z_{n}$ be nonzero values as specializations of some CMPLs at algebraic points. We define the weight of the monomial $Z_{1}^{m_{1}}\ldots Z_{n}^{m_{n}}$ to be
\[ \sum_{i=1}^{n}m_{i}{\rm{wt}}(Z_{i}) .\]
\end{definition}

Let $\overline{\mathfrak{M}}_{w}$ (resp. $\mathfrak{M}_{w}$) be the $\ok$-vector space (resp. $k$-vector space) spanned by the total weight $w$ monomials of CMPLs at algebraic points, and $\overline{\mathfrak{M}}$ (resp. $\mathfrak{M}$) be the $\ok$-algebra (resp. $k$-algebra) generated by all CMPLs at algebraic points. Note that the stuffle relation (cf.~\S~\ref{sub:stuffle}) implies that $\overline{\mathfrak{M}}_{w_{1}}\overline{\mathfrak{M}}_{w_{2}}\subset \overline{\mathfrak{M}}_{w_{1}+w_{2}}$. By applying Theorem~\ref{T:LinIndepPsitheta} we obtain the following result.
\begin{theorem}\label{T:Main Thm2}
Let $w_{1},\ldots,w_{\ell}$ be $\ell$ distinct positive integers. Let $V_{i}$ be a finite set consisting of weight $w_{i}$ monomials of some nonzero values as  specializations of Carlitz multiple polylogarithms at algebraic points for $i=1,\ldots,\ell$. If $V_{i}$ is a linearly independent set over $k$, then
\[ \left\{1  \right\} \bigcup_{i=1}^{\ell} V_{i}  \]
is linearly independent over $\bar{k}$. In particular, we have that
\begin{enumerate}
\item $\overline{\mathfrak{M}}$ is a graded algebra, i.e.,  $ \overline{\mathfrak{M}}=\bar{k}\oplus_{w\in \NN}\overline{\mathfrak{M}}_{w}$.
\item $\overline{\mathfrak{M}}$ is defined over $k$ in the sense that the canonical map $\ok\otimes_{k}\mathfrak{M}\rightarrow \overline{\mathfrak{M}}$ is bijective.
\end{enumerate}
\end{theorem}
\begin{proof}
By Proposition~\ref{P:CMPLMZ} each nonzero value as a specialization of Carlitz multiple polylogarithm at an algebraic point has the {\it{MZ}} property. It follows that by Proposition~\ref{P:monomialsMZ} each nontrivial monomial of such values has the {\it{MZ}} property. Therefore, the result follows from Theorem~\ref{T:LinIndepPsitheta}.
\end{proof}

\subsection{Application to MZVs}

\subsubsection{Review of Anderson-Thakur theory}

We put $D_{0}:=1$, and $D_{n}:=\prod_{i=0}^{n-1}( \theta^{q^{n}}-\theta^{q^{i}} )$ for $n\in \NN$. For any nonnegative integer $n$, we define the Carlitz factorial
\[ \Gamma_{n+1}:=\prod_{i}D_{i}^{n_{i}},  \]
where
\[ n=\sum n_{i}q^{i}\hbox{ }(0\leq n_{i}\leq q-1) \]
is the base $q$ expansion of $n$. The following theorem is the key ingredient to connect MZVs with CMPLs.
\begin{theorem}{\textnormal{(Anderson-Thakur, \cite[3.7.3, 3.7.4]{AT90} and \cite[2.4.1]{AT09})}}\label{T:Hs}
There exists a sequence of polynomials $H_{n}(t)\in A[t]$  $(n=0,1,2,\ldots)$ such that
\[ \left( H_{s-1}\Omega^{s} \right)^{(d)}(\theta)=\frac{\Gamma_{s}S_{d}(s)}{\tilde{\pi}^{s}}  \]
for all nonnegative integers $d$ and $s\in \NN$. Moreover, when one regards $H_{n}$ as a polynomial of $\theta$ over $\FF_{q}[t]$ then
\[ \deg_{\theta} H_{n} \leq \frac{n q}{q-1}.\]
\end{theorem}

\subsubsection{Connection between MZVs and CMPLs}
 The following result is a generalization of Anderson-Thakur \cite[\S\S~3.9]{AT90}.
\begin{theorem}\label{T:CMPLsMZVs}
 Given any $r$-tuple $\mathfrak{s}=(s_{1},\ldots,s_{r})\in \NN^{r}$, we let $H_{n}(t)\in A[t]$ be the polynomials in Theorem~\ref{T:Hs}. Let $S$ be the set of points $\bu=(u_{1},\ldots,u_{r})\in A^{r}$ with $u_{j}$ running over all coefficients of $H_{s_{j}-1}\in A[t]$ for all $j=1,\ldots,r$. Note that based on Theorem~\ref{T:Hs} any point $\bu$ of $S$ belongs to $\DD_{\mathfrak{s}}$.  For each $\bu=(u_{1},\ldots,u_{r})\in S$, let $u_{j}$ correspond to the coefficient of $t^{m_{j}}$ in $H_{s_{j}-1}$ and put
\[a_{\bu}=\prod_{j=1}^{r}\theta^{m_{j}}=\theta^{m_{1}+\cdots+m_{r}} . \] Then we have
\[\zeta_{A}(s_{1},\ldots,s_{r})=\frac{1}{\Gamma_{s_{1}}\cdots\Gamma_{s_{r}}} \sum_{\bu\in S} a_{\bu}\Li_{\mathfrak{s}}(\bu).\]
\end{theorem}
\begin{proof}
We first note that by Theorem~\ref{T:Hs} we have $\parallel H_{s_{j}-1}\parallel_{\infty}<{q}^{\frac{s_{j}q}{q-1}}$ for $j=1,\ldots,r$. We take \[\mathfrak{Q}:=\left(H_{s_{1}-1},H_{s_{2}-1},\ldots,H_{s_{r}-1}\right)\in A[t]^{r}\]
and so the series $L_{\mathfrak{s},\mathfrak{Q}}(t)$ defined in (\ref{E:LsQ}) is an entire function by Lemma~\ref{L:convergence}. By the definition of $L_{r+1}$ we have
\[
 L_{r+1}(t)=\Omega^{s_{1}+\cdots+s_{r}}(t)\sum_{i_{1}>\cdots>i_{r}\geq 0} \frac{H_{s_{r}-1}^{(i_{r})}(t)\cdots H_{s_{1}-1}^{(i_{1})}(t) }{ \left( (t-\theta^{q})\ldots(t-\theta^{q^{i_{r}}})  \right)^{s_{r}}\ldots \left((t-\theta^{q})\ldots(t-\theta^{q^{i_{1}}})  \right)^{s_{1}}         } .
 \]
So Theorem~\ref{T:Hs} implies that
\begin{equation}\label{E:Lr+1theta}
 L_{r+1}(\theta)=\frac{\Gamma_{s_{1}}\cdots\Gamma_{s_{r}}\zeta_{A}(s_{1},\ldots,s_{r})  }{\tilde{\pi}^{s_{1}+\cdots+s_{r}}   }.
 \end{equation}
We write
\[ L_{r+1}(t)/\Omega^{s_{1}+\cdots+s_{r}}(t)=\sum_{i_{1}>\cdots>i_{r}\geq 0} \frac{H_{s_{r}-1}^{(i_{r})}(t)\cdots H_{s_{1}-1}^{(i_{1})}(t) }{ \left( (t-\theta^{q})\ldots(t-\theta^{q^{i_{r}}})  \right)^{s_{r}}\ldots \left((t-\theta^{q})\ldots(t-\theta^{q^{i_{1}}})  \right)^{s_{1}}         } ,\]
and so combining with (\ref{E:Lr+1theta}) we have
\begin{equation}\label{E:GammaMZV}
\Gamma_{s_{1}}\cdots\Gamma_{s_{r}}\zeta_{A}(s_{1},\ldots,s_{r})=\sum_{i_{1}>\cdots>i_{r}\geq 0} \frac{H_{s_{r}-1}^{(i_{r})}(\theta)\cdots H_{s_{1}-1}^{(i_{1})}(\theta) }{ \left( (\theta-\theta^{q})\ldots(\theta-\theta^{q^{i_{r}}})  \right)^{s_{r}}\ldots \left((\theta-\theta^{q})\ldots(\theta-\theta^{q^{i_{1}}})  \right)^{s_{1}}         }.
\end{equation}

Note that the RHS of (\ref{E:GammaMZV}) equals $\sum_{\bu\in S}a_{\bu}\Li_{\mathfrak{s}}(\bu)$. By dividing $\Gamma_{s_{1}}\cdots\Gamma_{s_{r}}$ on the both sides of (\ref{E:GammaMZV}) we obtain the desired formula of $\zeta_{A}(s_{1},\ldots,s_{r})$.
\end{proof}

\begin{corollary}\label{Cor:DescentMZVs}
Given any $(s_{1},\ldots,s_{r})\in \NN^{r}$, we have that if $\zeta_{A}(s_{1},\ldots,s_{r})/\tilde{\pi}^{s_{1}+\cdots+s_{r}}\in \ok$, then $\zeta_{A}(s_{1},\ldots,s_{r})/\tilde{\pi}^{s_{1}+\cdots+s_{r}}\in k$.
\end{corollary}
\begin{proof}
Put $w:=s_{1}+\cdots+s_{r}$. By Theorem~\ref{T:CMPLsMZVs} $\zeta_{A}(s_{1},\ldots,s_{r})$ is a $k$-linear combination of CMPLs at algebraic points of weight $w$. The result follows by  Proposition~\ref{P:CMPLMZ} and Proposition~\ref{P:DescentMZ}.
\end{proof}

\subsubsection{Proof of Theorem~\ref{T:Main Thm}}\label{sub:ProofMainThm} Now we give a proof of Theorem~\ref{T:Main Thm}. By Theorem~\ref{T:CMPLsMZVs} we have $\mathcal{Z}_{w}\subset \mathfrak{M}_{w}$ for each $w\in \NN$. So Theorem~\ref{T:Main Thm} follows from Theorem~\ref{T:Main Thm2}.

\begin{remark} We mention that after the results of this paper were obtained in 2012,  recently  Mishiba~\cite{M14} studied the MZVs of \lq\lq odd\rq\rq coordinates with certain restrictions and proved some algebraic independence results using Papanikolas' theory.  We also note that the series (\ref{E:LsQ}) with  specific $\mathfrak{Q}$ was first studied in \cite{P08, CY07, AT09} and later on was studied by Mishiba for $\mathfrak{Q}$ as a vector of polynomial entries with certain restriction on norms. However, in \cite{M14} the entireness property of $L_{\mathfrak{s},\mathfrak{Q}}$ is not studied.
\end{remark}

\bibliographystyle{alpha}

\end{document}